\documentclass[a4paper, 11pt, reqno]{amsart}
\usepackage{amssymb}
\usepackage{amsthm}
\usepackage{bm}
\usepackage{latexsym}
\usepackage{amsmath}
\usepackage{eufrak}
\usepackage{mathrsfs}
\usepackage{caption}
\usepackage{amscd}
\usepackage[all,cmtip]{xy}
\usepackage[usenames]{color}
\usepackage{graphicx}
\usepackage{color}
\usepackage{amscd}

\theoremstyle{plain}

\newtheorem{theorem}{Theorem}[section]
\newtheorem{proposition}[theorem]{Proposition}
\newtheorem{lemma}[theorem]{Lemma}
\newtheorem{corollary}[theorem]{Corollary}
\newtheorem{conjecture}[theorem]{Conjecture}

\theoremstyle{definition}

\newcommand{\appsection}[1]{\let\oldthesection\thesection
\renewcommand{\thesection}{Appendix \oldthesection}
\section{#1}\let\thesection\oldthesection}

\newtheorem{definition}[theorem]{Definition}

\theoremstyle{remark}

\newtheorem{remark}[theorem]{Remark}
\newtheorem{example}[theorem]{Example}

\def\R{{\mathbb{R}}}
\def\Z{{\mathbb{Z}}}
\def\F{{\mathbb{F}}}
\def\Q{{\mathbb{Q}}}
\def\C{{\mathbb{C}}}
\def\K{{\mathit{k}}}
\def\P{{\mathbb{P}}}

\def\AR{{\mathcal{A}}}

\begin{document}

\title{On the geography of line arrangements}
\author[Sebastian Eterovi\'c]{Sebastian Eterovi\'c}
\email{eterovic@maths.ox.ac.uk}
\address{Mathematical Institute, University of Oxford, Andrew Wiles Building, Radcliffe Observatory Quarter, Woodstock Road, Oxford, UK.}
\author[Fernando Figueroa]{Fernando Figueroa}
\email{fefigueroa@uc.cl}
\address{Facultad de Matem\'aticas, Pontificia Universidad Cat\'olica de Chile, Campus San Joaqu\'in, Avenida Vicu\~na Mackenna 4860, Santiago, Chile.}
\author[Giancarlo Urz\'ua]{Giancarlo Urz\'ua}
\email{urzua@mat.uc.cl}
\address{Facultad de Matem\'aticas, Pontificia Universidad Cat\'olica de Chile, Campus San Joaqu\'in, Avenida Vicu\~na Mackenna 4860, Santiago, Chile.}
%
%
%
\date{\today}

\begin{abstract}
This is a short note on various results about the combinatorial properties of line arrangements in terms of the Chern numbers of the corresponding log surfaces. This resembles the study of the geography of surfaces of general type. We prove some new results about the distribution of Chern slopes, we prove a connection between their accumulation points and the accumulation points of linear $H$-constants on the plane, and we present two open problems in relation to geography over $\Q$ and over $\C$. 

\end{abstract}

\maketitle

\section{Introduction} \label{intro}

Let $\K$ be an arbitrary field. The projective plane over $\K$ will be denoted by $\P^2_{\K}$. Our motivating question is whether one can describe the behaviour of the Chern numbers of line arrangements on $\P^2_{\K}$ (Chern numbers will be introduced in \textsection\ref{s2}). The Chern numbers of line arrangements were first introduced by Hirzebruch in \cite[\textsection 3.3]{Hirz83} to study the Chern numbers of algebraic surfaces of general type. However, one can circumvent this and define the Chern numbers of line arrangements purely in terms of the incidence structure they define (this is the approach taken here). As such, Chern numbers are susceptible to certain combinatorial properties of line arrangements. By recalling some of the most important results about these numbers and also presenting some new results (in particular, we will show a connection with the so-called linear $H$-constants in \textsection \ref{s3}), we aim to show that the study of Chern numbers of line arrangements proves to be insightful both from a combinatorial and from a geometric perspective, and leads to very interesting questions which we will present at the end.

The most relevant fields for us will be $\Q$, $\R$, $\C$, and the algebraic closure $\overline{\F}_p$ of the field of $p$ elements $\F_p$.


\section{Definitions, examples, and combinatorial facts} \label{s1}

\begin{definition}
A set of the form $\{[x,y,z] \in \P^2_{\K} \colon ax+by+cz=0\}$ for some $a,b,c \in \K$ not all zero, will be called a \textit{line}. A \textit{line arrangement} is a finite collection of two or more lines.
\label{lines} 
\end{definition}

\begin{definition}
The \textit{incidence structure} of a line arrangement $\AR$ is the data $(\mathcal{P},\mathcal{L}, I)$, where $\mathcal{L}$ is the set of lines of $\AR$, $\mathcal{P}$ is the set of points belonging to at least two lines in $\mathcal{L}$, and $I\subseteq \mathcal{P}\times \mathcal{L}$ is the incidence relation saying which points belong to which lines. For $m\geq 2$, an \textit{m-point} of $\AR$ is a point in $\mathcal{P}$ which belongs to exactly $m$ lines in $\mathcal{L}$. We denote the number of $m$-points by $t_m$. A line arrangement is said to be in \textit{general position} if the arrangement satisfies $t_m=0$ for every $m>2$ (so it only has double points).
\label{comb}
\end{definition}

Our main reference on line arrangements is \cite{Hirz83}. 

\begin{example}
An arrangement of $d$ lines with $t_d=1$ is called \textit{trivial}. It consists of $d$ concurrent lines, and so $t_{k}=0$ for $2\leq k < d$ (this arrangement is also sometimes called a \textit{pencil of lines}. An arrangement is called \textit{quasi-trivial} (or a \textit{near-pencil}) if $t_{d-1}=1$. This consists of $d-1$ lines meeting at the same point $p$, and another line not going through $p$. In this case $t_{2} = d-1$, $t_{d-1} = 1$ and $t_{k}=0$ otherwise. Given that we understand these two types of arrangements, we will not consider them once we introduce Chern numbers.
\label{trivial}
\end{example}

\begin{example}[Real line arrangements (see {{\cite[1.1]{Hirz83}}})]
Line arrangements in the real projective plane partition $\P^2_{\R}$ into polygons. If all polygons are triangles, then the arrangement is called \textit{simplicial}. There is a vast literature on simplicial arrangements (cf. \cite{Gr05}). They have not been classified yet. An example is the \textit{complete quadrilateral} defined by the zeros of $xyz(x-y)(x-z)(y-z)$. It is an arrangement of $6$ lines with $t_2=3$, $t_3=4$, $t_m=0$ else. Regular polygons define families of simplicial arrangements: by taking the regular polygon of $n$ lines and adding its $n$ lines of symmetry, we get an arrangement of $2n$ lines. It has $t_2 = n$, $t_3 = n(n-1)/2$, $t_n = 1$, $t_m=0$ else.
\label{simplicial}
\end{example}

\begin{example}
Let $n\geq 4$. Then the zeros of $(x^n-y^n)(x^n-z^n)(y^n-z^n)$ in $\P^2_{\C}$ define an arrangement of $3n$ lines with $t_3=n^2$, $t_n=3$, $t_m=0$ else. They are called \textit{Ceva arrangements}. For $n=3$, the polynomial $(x^3-y^3)(x^3-z^3)(y^3-z^3)$ defines what is known as the \textit{dual Hesse arrangement} which has $9$ lines and $12$ triple points, that is to say $d=9$, $t_3=12$, and $t_m=0$ otherwise. The \textit{Hesse arrangement} is the arrangement of $12$ lines joining the $9$ inflection points of a given smooth projective cubic in $\P^2_{\C}$. It turns out that they are all projectively equivalent, and they have $t_2=12$, $t_4=9$, $t_m=0$ else. The dual lines defined by the nine 4-points are the $9$ lines of the dual Hesse arrangement. We recall that points and lines are dual objects of each other, in the sense that an arrangement of lines corresponds to the collection of points in $\P^2_{\K}$ given by the $3$ coefficients of each line, and vice-versa.
\label{Ceva}
\end{example}

\begin{example}
Let $\K=\F_{p^n}$ for some prime $p$ and $n>0$. The set of $p^{2n}+p^n+1$ lines in $\P^2_{\K}$ form an arrangement of lines with $t_{p^n+1}=p^{2n}+p^n+1$, $t_m=0$ else. We call it a \textit{finite projective plane arrangement}. For $p=2$ and $n=1$ we have the \textit{Fano arrangement} of seven lines with seven triple points. 
\label{finite}
\end{example}

By counting pairs of lines in two different ways, we obtain that any arrangement of $d$ lines satisfies $$ {d \choose 2} = \sum_{m \geq 2} {m \choose 2} t_m,$$ which is a purely combinatorial fact. Another general statement is the following theorem, originally proven in \cite{deBrEr48}, and of which there exist many purely combinatorial proofs (see e.g. \cite[Theorem 14.1.13]{IoSh06} and \cite[\textsection 12.2]{Ju11}; the latter presents an argument due to Conway). We present another proof which, although not as general as the ones we have cited, takes advantage of the field structure underlying $\P_{\K}^{2}$ to explicitly recover the finite field over which the arrangement exists, and shows how to interpret the field operations in terms of intersections of lines.

\begin{theorem}
A nontrivial arrangement of $d$ lines $\AR$ satisfies $$\sum_{m \geq 2} t_m \geq d.$$ Equality holds if and only if $\AR$ is either quasi-trivial or a finite projective plane arrangement.
\label{erdos}
\end{theorem} 

\begin{proof}
This first part is taken from \cite[Remark 7.4]{Urz11}. Let us label the $m$-points of the arrangement from $1$ to $r=\sum_{m \geq 2} t_m $, and the lines from $1$ to $d$. We define $$a_{i,j}= \left\{ \begin{array}{lcc}
             1 &   \text{if the line } j \text{ contains point }i \\
             \\ 0  &  \text{otherwise.} 
             \end{array}
   \right.$$ Let $L_j$ be the vector $(a_{i,j})_{1 \leq i \leq r}$. We want to prove that the $L_j$ are linearly independent in $\Q^r$. Suppose not, say that $L_1=\sum_{j=2}^d x_j L_j$ for some $x_j \in \Q$. Then, by taking the usual inner product in $\Q^r$, we have $$x_j = \frac{L_1 \cdot L_1-1}{1-L_j \cdot L_j} <0 $$ for all $j>1$. But the coordinates of $L_1$ are either 1 or 0, and so it is impossible that every $x_{j}$ is negative. This proves the inequality part of the statement.

For the second part of the statement, observe first that a quasi-trivial arrangement and a finite projective plane arrangement have $r=d$. Conversely, assume that $\AR$ satisfies $r=d$ and that it is not quasi-trivial. Define the matrix $A=(L_j)_{1\leq j \leq d}$, and let $n_j$ be the number of points in the line $L_j$, and let $g_i$ be the number of lines passing through the $i$-th point. We now use \cite[Section 1]{HP79}.

We first prove the following lemma:

\begin{lemma}
There is a permutation matrix $P$ such that $(PA)_{ii}=0$ for every $1\leq i \leq d$.
\end{lemma}
\begin{proof}

Define the matrix $T$; $T_{ij}=1,1 \leq i \leq d, 1 \leq j \leq d$, so \begin{align*}
A^t(T-A)=A^tT-A^tA= &
\begin{bmatrix}
 0 & n_1 -1 & n_1-1  & \ldots & n_1-1\\
 n_2 -1 & 0 & n_2-1 & \ldots  & n_2-1\\
 \vdots & \vdots & \vdots & \ddots & \vdots\\
 n_d-1 & n_d-1 & n_d-1 & \dots & 0 
\end{bmatrix}.
\end{align*}

Thus we have $$det(A^t(T-A))=\prod_{j=1}^d(n_j-1)det\begin{bmatrix}
 0 & 1 & 1  & \ldots & 1\\
 1 & 0 & 1 & \ldots  & 1\\
 \vdots & \vdots & \vdots & \ddots & \vdots\\
 1 & 1 & 1 & \dots & 0 
\end{bmatrix} \neq 0.$$ Then $T-A$ is a matrix with entries in $\{0,1\}$ and non zero determinant, hence there is an addend in the determinant formula which is non zero, meaning in our case that it is the multiplication of only ones. Therefore there must be a row permutation $P$ such that $P(T-A)$ has only $1$ in its diagonal, and so $PA$ has only $0$ on its diagonal.
\end{proof}

Hence by reordering the $m$-points, we can assume that $A$ has $a_{kk}=0$, that is, the line $L_k$ does not contain the $k$-th point. In this way we have that $n_k \geq g_k$ for every $k$. But then $\sum_j n_j=\sum_i\sum_j a_{ij}=\sum_i g_i$ implies $n_k=g_k$ for every $i$. This in turn implies that there is a line passing through any pair of points. Indeed this gives the second equality in the count 
\begin{align*}
& \# \{(h,k):\text{there is a line containing point h and point k} \} \\
& =\sum_{i=1}^d {n_i \choose 2}=\sum_{i=1}^d {g_i \choose 2}=\sum_{m \geq 2} {m \choose 2} t_m={d \choose 2}.
\end{align*}

Finally, unless $\AR$ is a quasi-trivial arrangement, there are $4$ points in the arrangement, with no $3$ in a line, then we get that for every pair of points there is a line not passing through either of them. Let $P_i$ and $P_j$ be two points, and let $L_k$ be a line not containing them. Then $g_i=n_k$ and $g_j=n_k$. Therefore $g_1=\ldots=g_d$, and so $n_j=g_i$ for every $i \neq j$. 

Assume now that $n_i=g_j=q+1$ for all $i,j$, for some $q$. So we have that $d=q^2+q+1=t_{q+1}$, $t_m=0$ else. We recall that our line arrangement $\AR$ is in $\P_{\K}^2$ for some field $\K$. 

By a change of coordinates, we can assume that four of the $(q+1)$-points of $\AR$ are $[1,0,1]$, $[0,1,1]$,$[1,0,0]$, and $[0,1,0]$. Hence we also have the $(q+1)$-point $[0,0,1]$, and the lines $x=0$, $z=0$, $x=z$, and $y=z$.

Let $\K'= \{x \in \K \ \text{such that} \ [1,x,0] \ \text{is a} \ (q+1)\text{-point} \}$. We will show that this set is a subfield of $\K$. Clearly $0,1\in \K'$. Since we know it has exactly $q$ elements, we only need to show that $\K'$ is a ring. From now on, we will use the notation $[a,b,c]-[c,d,e]$ for the line through the points $[a,b,c]$ and $[c,d,e]$. 

We have $[1,c,0]-[1,0,1] \cap \{x=0\}=[0,-c,1]$, $[0,-c,1]-[1,0,0] \cap \{x=z\}=[1,-c,1]$, and $[1,-c,1]-[0,0,1] \cap \{z=0\} = [1,-c,0]$. We also have $[1,c,0]-[0,1,1] \cap \{x=z\} =[1,c+1,1]$, and $[1,c+1,1]-[0,0,1] \cap \{z=0\} = [1,c+1,0]$. Thus if $c \in \K'$, then $c+1 \in \K'$ and $-c \in \K'$. By interchanging the roles of $x$ and $y$, we get that if $[c,1,0]$ is a $(q+1)$-point, then so is $[c+1,1,1]$.

Let $a,b \in \K'$ with $a \neq b$, $b \neq 0$. Then $[1,a,0]$ and$[1,b,0]$ are $(q+1)$-points, and by the previous paragraph we know that this implies that $[1,a+1,1]$ and $[b^{-1}+1,1,1]$ are $(q+1)$-points. Observe that $[1,a+1,1]-[b^{-1}+1,1,1] \cap \{z=0\}=[1,-ab,0] $. Thus $ab \in \K'$. 

We now want to show that $c \in \K'$ implies $c^2 \in \K'$. Say $c \neq 0,1,-1$. If $c^m=1$ for $m<5$, then $c^2$ is $1$, $-c-1$ or $-1$, already in $k'$. Note that $c(c+1)=c^2+c \in \K'$, and so $c^2+c+1 \in \K'$. Also $c^2+c+1 \neq c-1$ (or we would already have $c^2=-2 \in k'$), and so $(c-1)(c^2+c+1)=c^3-1 \in \K'$. Thus $c^3 \in \K'$. As $c^3\neq c$, then $c^4 \in k'$. As $k'$ is finite there must be an $m$ such that $c^m=1$. Multiplying $c^4$ by $c$ enough times, we get $c^{m-1}\in k'$, and then $c^{m-1}c^3=c^2 \in k'$.

In this way, because $\K'$ has finite order, given a nonzero $c \in \K'$, we have $c^{-1} \in \K'$. Therefore given $a,b \in \K'$, we have that $a+b=b(ab^{-1}+1) \in \K'$. This completes the proof that $\K'=\F_{q}$, and so $\AR$ is projectively equivalent to a copy of the finite projective plane arrangement $\P_{\F_q}^2$ in $\P_{\K}^2$. 
\end{proof}

\begin{remark}
One can ask if there are other situations in which incidence structures satisfying the condition ``the number of points equals the number of block''  can be realised as the incidence structure of some curve arrangement. In \cite{E15}, the first author showed that the incidence structure of certain Ryser designs (see \cite[\textsection 14]{IoSh06} for definitions) are realised as arrangements of curves in Hirzebruch surfaces over a finite field. 
\end{remark}

\section{Chern Numbers} \label{s2}

We now define the key combinatorial invariants for line arrangements that we will study. Surprisingly, various general properties of line arrangements can be expressed with these invariants. 

\begin{definition}
Let $\AR$ be an arrangement of $d$ lines. We define the integers 
$$\bar{c}_1^2(\AR)= 9-5d+\sum_{m \geq 2} (3m-4) t_m \ \ \ \ \text{and} \ \ \ \ \bar{c}_2(\AR)= 3-2d + \sum_{m \geq 2} (m-1) t_m.$$ They are called the \textit{Chern numbers} of $\AR$.
\label{chern}
\end{definition}

\begin{remark}
We have defined the Chern numbers of line arrangements only in terms of their incidence structure, and as such, the definition may seem arbitrary. It becomes more natural if we see these invariants using Hirzebruch's original construction, which we summarise next (see \cite{Hirz83} or \cite[Chapter 5]{Tr16} for full details). Let $\sigma \colon X \to \P^2$ be the blow-up of all the $m$-points of a line arrangement, with $m>2$. Let $D$ be the reduced total transform of the arrangement under $\sigma$, and so it contains all strict transforms of the lines and all exceptional divisors of $\sigma$. Let $\Omega_X^1(\log D)$ be the rank two vector bundle on $X$ of log differentials with poles in $D$. Let $c_i(\Omega_X^1(\log D)^{*})$, $i=1,2$, be the Chern classes of the dual of $\Omega_X^1(\log D)$ (see \cite[\textsection 1.4 and \textsection 3.2]{Tr16}). Now we define the Chern numbers of the line arrangement in terms of these Chern classes as: $\bar{c}_1^2=c_1 \cdot c_1$ and $\bar{c}_2= c_2$. See also \cite[\textsection 2 and \textsection 4]{U10a}, where this process is done in more generality for arrangements of curves in algebraic surfaces. 
\label{logdiff}
\end{remark}

\begin{proposition}
If $\AR$ has $t_d=t_{d-1}=0$, then its Chern numbers are positive.
\label{positive}
\end{proposition}

\begin{proof}
We note that a quasi-trivial arrangement has $\bar{c}_1^2=\bar{c}_2=0$. Suppose $d=4$. Then $\AR$ only has nodes (under the conditions of the proposition) and therefore $\bar{c}_1^2(\AR)=1$ and $\bar{c}_2(\AR)=1$. We now argue by induction on $d$. Assume that $\AR$ has $d+1 \geq 5$ lines, and let $L \in \AR$ be a line passing by $t \geq 3$ points (it must exist by the assumptions). The arrangement $\AR \setminus L$ is not trivial, and so $$\bar{c}_1^2(\AR) \geq \bar{c}_1^2(\AR \setminus L)-5 +2t \geq \bar{c}_1^2(\AR \setminus L)+1 \geq 1$$ and $$\bar{c}_2(\AR)= \bar{c}_2(\AR \setminus L)-2 +t \geq \bar{c}_1^2(\AR \setminus L)+1 \geq 1.$$
\end{proof}

\begin{proposition} (see \cite[Theorem (5.1)]{So84}) Let $\AR$ be an arrangement of $d$ lines such that $t_d=t_{d-1}=0$. Then,
$$\frac{2d-6}{d-2}\leq \frac{\bar{c}_1^2}{\bar{c}_2} \leq 3.$$ Left equality holds if and only if $t_2={d \choose 2}$ (i.e. the arrangement has only nodes), and right equality holds if and only if $\sum_{m \geq 2} t_m=d$ (and so $\AR$ is a finite projective plane arrangement).
\label{combineq}
\end{proposition}

\begin{proof}
The left inequality is equivalent to 
$$0\leq(d-2)\bar{c}_1^2-(2d-6)\bar{c}_2 = \sum_{m \geq 2}^{d-2} t_m(-m^2+m(1+d)+(2-2d))$$ but $-m^2+m(1+d)+(2-2d) \geq 0$ for all $2 \leq m \leq d-1$. Moreover we have $-m^2+m(1+d)+(2-2d) > 0$ for all $3 \leq m \leq d-2$. But recall that by hypothesis $t_d=t_{d-1}=0$. This proves the first inequality. 

The second inequality is equivalent to showing that $$\bar{c}_1^2-3\bar{c}_2 = d - \sum_{m \geq 2} t_m\leq 0,$$ but this follows from Theorem \ref{erdos}.
 
\end{proof}

All statements about Chern numbers so far have been proven combinatorially, without referencing the ground field $\K$. The next theorem shows the one can strengthen Proposition \ref{combineq} when $\K=\R$ and $\K=\C$.  

\begin{theorem}
Let $\AR$ be an arrangement of $d$ lines with $t_d=t_{d-1}=0$.
\begin{itemize}
\item[1)] If $\K=\R$, then $\bar{c}_1^2 \leq \frac{5}{2} \bar{c}_2$. Equality is achieved if and only if $\AR$ is simplicial (see Example \ref{simplicial}).

\item[2)] If $\K=\C$, then $\bar{c}_1^2 \leq \frac{8}{3} \bar{c}_2$. Equality is achieved if and only if $\AR$ is the dual Hesse arrangement (see Example \ref{Ceva}).
\end{itemize}
\label{HS}
\end{theorem}

\begin{proof}
We follow \cite[p.115]{Hirz83} for the proof of 1). As we noted in Example \ref{simplicial}, a real arrangement partitions $\P_{\R}^2$ in polygons. This can be used to compute the topological Euler characteristic of $\P_{\R}^2$ which is equal to $1$. With that one obtains $$\sum_{m \geq 3} (m-3) p_m = -3 - \sum_{m \geq 2} (m-3)t_m,$$ where $p_m$ is the number of $m$-gons. On the other hand, one can check that $$5 \bar{c}_2 -2 \bar{c}_1^2 = -3 - \sum_{m \geq 2} (m-3)t_m,$$ and so we get what we want for 1).

The claim in 2) is essentially the Hirzebruch-Sakai inequality \cite{Hirz83}, which comes form the Bogomolov-Miyaoka-Yau inequality for algebraic surfaces. See \cite{Hirz83}, \cite[Chapter 4]{Tr16}, \cite[Theorem 5.3]{So84}, \cite[Proposition II.8]{U08} for details.
\end{proof}

\section{Density of Chern slopes} \label{s3}

Analogous to the geography problem for surfaces of general type (cf. \cite{P87}), we can talk about the geography problem for line arrangements over a fixed field $\K$: Given $(a,b) \in \Z^2$, is there a line arrangement over $\K$ with $\bar{c}_1^2=a$ and $\bar{c}_2=b$? From now on, we will restrict all line arrangements of $d$ lines to satisfy $t_d=t_{d-1}=0$. We recall that for every $\K$  $$2-\frac{2}{d-2}\leq \frac{\bar{c}_1^2}{\bar{c}_2} \leq 3$$ by Proposition \ref{combineq}. The geography problem for line arrangements could be hard to solve in general. A slightly easier variant of the geography problem is to ask: What positive rational numbers can appear as the quotient $\frac{\bar{c}_1^2}{\bar{c}_2}$ of a line arrangement? Our focus in this section is to obtain constraints for the possible values of the \textit{Chern slope} $\frac{\bar{c}_1^2}{\bar{c}_2}$ for a fixed $\K$. For example, we have already seen that $\frac{\bar{c}_1^2}{\bar{c}_2}=3$ can only be realised by a finite projective plane arrangement, or that over $\C$ the only line arrangement satisfying $\frac{\bar{c}_1^2}{\bar{c}_2}=\frac{8}{3}$ is the dual Hesse arrangement. Aside from the results we already have about specific values of the Chern slope, our goal now is to find all accumulation points of Chern slopes for a given field $\K$. We start with a simple corollary of Proposition \ref{combineq}.

\begin{corollary}
If $r$ is an accumulation point of Chern slopes, then $r \in [2,3]$. 
\label{acc}
\end{corollary}   

\begin{proof}
By Proposition \ref{combineq} we know that $1\leq \frac{\bar{c}_1^2}{\bar{c}_2}\leq 3$. We note that after fixing the number of lines there are only finitely many different combinatorial arrangements, and so finitely many possible Chern slopes. Let $s$ denote the Chern slope of some line arrangement. If $s < 2$, then $s < \frac{2d-6}{d-2}$ for only finitely many $d$. But we know that an arrangement of $d$ lines satisfies $\frac{2d-6}{d-2}\leq \frac{\bar{c}_1^2}{\bar{c}_2}$ by Proposition \ref{combineq}, and so the Chern slopes cannot accumulate below 2.
\end{proof}

The following is inspired by the density lemma in \cite[Lemma 11.1]{E15}. 

\begin{lemma}
Let $\K$ be an infinite field. Let $\AR_n$ be a collection of arrangements of $l(n)$ lines over a field $\K$ with $\lim_{n \to \infty} \frac{\bar{c}_1^2}{\bar{c}_2} = c >2$ and $\lim_{n \to \infty} l(n)= \infty$. Assume there is $h \in ]1,2]$ such that $\lim_{n \to \infty} \frac{\bar{c}_1^2}{l(n)^h}=a>0$. Then Chern slopes of line arrangements over $\K$ are dense in $[2,c]$.  
\label{density}
\end{lemma}

\begin{proof}
Let $x \in \R_{>0}$. We choose $n_0 \gg 0$ such that $d(n)=[x l(n)^{h-1}]$ are positive integers for all $n> n_0$, where $[y]$ is the integral part of $y$. 

For $n>n_0$, we consider the arrangements of $d(n)+l(n)$ lines $\AR'_n$ over $\K$ defined as $\AR_n$ together with $d(n)$ general lines, this is, $d(n)$ lines which add only nodes and no other $m$-points to $\AR_n$. Then $$\frac{\bar{c}_1^2(\AR'_n)}{\bar{c}_2(\AR'_n)}= \frac{\bar{c}_1^2(\AR_n) + 2 l(n)d(n)+ d(n)^2 - 6 d(n)}{\bar{c}_2(\AR_n) + l(n)d(n) + \frac{d(n)^2}{2}- \frac{5d(n)}{2}},$$ and so $$\frac{\bar{c}_1^2(\AR'_n)}{\bar{c}_2(\AR'_n)}= \frac{\frac{\bar{c}_1^2(\AR_n)}{l(n)^h} + 2 \frac{d(n)}{l(n)^{h-1}}+ \frac{d(n)^2}{l(n)^h} - 6 \frac{d(n)}{l(n)^h}}{\frac{\bar{c}_2(\AR_n)}{l(n)^h} + \frac{d(n)}{l(n)^{h-1}} + \frac{d(n)^2}{2l(n)}- \frac{5d(n)}{2l(n)}}.$$

Then, if $h <2$,  we have $\lim_{n \to \infty} \frac{\bar{c}_1^2(\AR'_n)}{\bar{c}_2(\AR'_n)} = \frac{a+2x}{\frac{a}{c} +x}=f(x)$, and if $h=2$, we get $\lim_{n \to \infty} \frac{\bar{c}_1^2(\AR'_n)}{\bar{c}_2(\AR'_n)} = \frac{a+2x+x^2}{\frac{a}{c} +x + \frac{x^2}{2}}=g(x)$. We note that both real functions $f(x)$ and $g(x)$ are continuous in $\R_{>0}$, and their range is $]2,c[$.    
\end{proof}

With this lemma and Corollary \ref{acc} we can give a complete description of the geography problem for line arrangements over $\overline{\F}_p$ and $\R$.

\begin{corollary}
The set of accumulation points of Chern slopes of arrangements over $\overline{\F}_p$ is the interval $[2,3]$. 
\label{charp}
\end{corollary}  

\begin{proof}
We apply Lemma \ref{density} for the collection $\AR_n$ of finite projective plane arrangements (Example \ref{finite}) given by $\P^2_{\F_{p^n}}$, where $l(n)=p^{2n} + p^n +1$. Here $c=3$ and we use $h=\frac{3}{2}$.   
\end{proof}

\begin{corollary}
The set of accumulation points of Chern slopes of arrangements over $\R$ is the interval $[2,\frac{5}{2}]$. 
\label{real}
\end{corollary} 

\begin{proof}
We apply Lemma \ref{density} for the collection $\AR_n$ of arrangements of $2n$ lines given by regular polygons of $n$ sides (see Example \ref{simplicial}). Here $c=\frac{5}{2}$ and we take $h=2$. 
\end{proof}

We note that the simplicial arrangements given by regular polygons are not defined over $\Q$ in general. This is because all realizations are projectively equivalent strictly over $\R$, and for $n>6$ the regular $n$-gon is not defined by lines over $\Q$. See \cite[Theorem 3.6]{Cun11} for details.
 
With respect to Chern slopes, the highest family for line arrangements defined over $\Q$, we can produce is the following:

\begin{example}
For any $n\geq 3$, consider the lines $\{y=\alpha z/2\}$, $\{x=\alpha z/2\}$, $\{y=x+(\beta-n+1)z\}$, and $\{y=-x+(\beta+1)z\}$, with $\alpha$ and $\beta$ sweeping all non negative integers up to $2n$ and $2n-2$, respectively. This is an $n$ by $n$ array of ``right triangle arrangements" of $8n$ lines. We note that when ``$n=\infty$" we get an infinite simplicial arrangement in $\R^2$ with only right isosceles triangles. For a fixed $n$, it has $t_2=6n^2+6n-8$, $t_3=2n^2-6n+8$, $t_4=2n^2+2n-3$, $t_{2n-1}=2$, $t_{2n+1}=2$, $t_m=0$ else. Hence its Chern Slope is equal to $$\frac{38n^2 - 18n-7}{16n^2 - 8n - 2},$$ which converges to $2.375$ as $n$ tends to infinity.
\label{record}
\end{example}

By Lemma \ref{density}, Example \ref{record} implies that over $\Q$, any $r \in [2, 2.375]$ is an accumulation point of Chern Slopes.

\subsection{H-constants}

We now turn to an interesting connection between the Chern slopes of line arrangements, and the linear $H$-constants. $H$-constants were first introduced in \cite{Harb} to study the bounded negativity conjecture on blow-ups of the complex projective plane. The \textit{linear H-constant} for a line arrangement $\AR$ is defined as $$H_L(\AR):= \frac{d^2-\sum_{m\geq 2} m^2 t_m}{\sum_{m\geq 2} t_m},$$ or equivalently $$H_L(\AR) = \frac{3-(\bar{c}_1^2-2\bar{c}_2)}{d-(\bar{c}_1^2-3\bar{c}_2)} -2.$$

As we will show, the limit points of $\bar{c}_1^2/\bar{c}_2$ are in one to one correspondence with the accumulation points of $H_L$. In the proof we need to take care of asymptotically trivial families, which we define below.

\begin{definition}
An infinite collection of arrangements of $d_n$ lines $\{\AR_{d_n}\}$ is \textit{asymptotically trivial} if $d_n$ tends to infinity and there are integers $n_0, D>0$ such that for $n>n_0$ we have the disjoint union $\AR_{d_n}= \AR' \cup \AR''$ with $\AR'$ with at most $D$ lines and $\AR''$ trivial arrangement.
\label{chern}
\end{definition}

\begin{proposition}
Let $\{\AR_{d_n}\}$ be an infinite collection of line arrangements with $d_n \to \infty$. Assume they are not trivial or quasi-trivial. Then $\bar{c}_2 \to \infty$. 
\label{c2}
\end{proposition}

\begin{proof}
Let $m_n$ be the maximum $m$ for an $m$-point in $\AR_{d_n}$. If $m_n=2$ for infinitely many $n$, then $\bar{c}_2 \to \infty$ for those $n$'s. If there is a sequence of $n$'s for which $m_n>2$, we have $$\bar{c}_2(\AR_{d_n}) \geq (2-m_n)+(m_n-2)(d_n-m_n)=(m_n-2)(d_n-m_n-1)$$ where $m_n-2>0$ and $d_n-m_n-1>0$ since they are not trivial or quasi-trivial. Then if $m_n \to \infty$ we are done. Otherwise $d_n \to \infty$, and we are done.   
\end{proof}

\begin{proposition}
Let $\{\AR_{d_n}\}$ be an infinite collection of line arrangements with $d_n \to \infty$, which has no asymptotically trivial sub-collection. Then $\frac{\bar{c}_2}{d_n} \to \infty$. 
\label{c2overd}
\end{proposition}

\begin{proof}
We note that it is enough to show $$\frac{\sum_{m\geq 2} (m-1)t_m}{d_n} \to \infty.$$
Let $m_n$ be the maximum $m$ for an $m$-point in $\AR_{d_n}$. Then we have $$\frac{d_n-1}{m_n} \leq \frac{\sum_{m\geq 2} (m-1)t_m}{d_n}, $$ and so, if $m_n <M$ for some $M$, then we are done. Otherwise there are subsequences with $m_n \to \infty$. Let us consider one such subsequence. 

Let $P_n \in \AR_{d_n}$ be a point realizing $m_n$. Then by our assumption $d_n-m_n \to \infty$ since there are no asymptotically trivial subsequences. Then $\AR_{d_n}=\AR' \cup \AR''$ with $\AR'$ having $d_n-m_n \to \infty$ lines, and $\AR''$ is the trivial arrangement through $P_n$. Therefore $$\frac{\sum_{m\geq 2} (m-1) t_m}{d_n} \geq \frac{(m_n-1)+(d_n-m_n) m_n}{d_n} \geq  \frac{(d_n-m_n) m_n}{(d_n-m_n) + m_n} \to \infty.$$ 
\end{proof}

One can check that for asymptotically trivial arrangements we have $$\bar{c}_1^2/\bar{c}_2 \to 2 \ \ \ \text{and} \ \ \ H_L \to -2.$$

The proofs of the next proposition are direct from the connection with $\bar{c}_1^2/\bar{c}_2$. 

\begin{proposition} $\quad$
\begin{enumerate}
\item If $r$ is an accumulation point for $H_L$, then $r \in [-2,-\infty[$.
\item For a fixed char $p >0$, the set of accumulation points for $H_L$ is $[-2,-\infty[$. 
\item For line arrangements over $\R$, the set of accumulation points for $H_L$ is $[-2,-3]$.
\end{enumerate}
\end{proposition}

We end by recalling that \cite[Theorem 3.3]{Harb} says that for line arrangements over $\mathbb{C}$, $H_{L}\geq -4$, and that $-4$ would be a limit value of $H_{L}$ if there existed a family of line arrangements on $\mathbb{C}$ whose Chern slopes converge to $8/3$.

\section{Open Problems}

We end with the main questions left to answer regarding the density of Chern slopes.

\begin{conjecture}
The set of accumulation points of Chern slopes of arrangements over $\C$ is $[2,\frac{5}{2}]$.
\label{complex}
\end{conjecture}

In terms of $H$-constants, Conjecture \ref{complex} translates to:

\begin{conjecture}
The set of accumulation points of $H_L$ for line arrangements over $\C$ is $[-2,-3]$.
\label{complexH}
\end{conjecture}

As evidence for these conjectures, we recall the following observation. 

\begin{remark}
As explained in \cite[\textsection(1.2)]{Hirz83}, one way of obtaining line arrangements in $\P^{2}_\C$ is through finite reflection groups. A complete study can be found in \cite{OrSo82}. Hirzebruch includes tables with the values for the $t_{m}$ in all cases. Among the examples obtained in this way is the Ceva family, along with other examples of line arrangement whose Chern slope is bigger than $\frac{5}{2}$. The important thing is that we know all possible line arrangements coming from finite unitary reflection groups, and they contain no family of line arrangements with Chern slope converging to something bigger that $\frac{5}{2}$. 
\end{remark}

We point out that in order to understand the behaviour of the Chern slopes of complex line arrangements, it suffices to understand the behaviour of those arrangements defined over $\overline{\Q}$. This is because for every complex line arrangement, there exists a line arrangement defined over $\overline{\Q}$ which has the same incidence structure. Indeed, given a complex line arrangement, we can write down an affine system of polynomial equations defined over $\Q$ describing the way in which the lines intersect. For this, regard the coefficients of the lines as variables, and $3\times 3$ determinants equal to zero as equations declaring concurrence of $3$ lines. But we also need to say that some lines do not concur. For that we introduce extra variables to multiply these determinants, so that we impose that these multiplications are equal to $1$. To avoid homogeneous issues, we declare from the beginning that the lines do not contain $[1,0,0]$. Thus we obtain a finite set of polynomials defined over $\Q$ describing the combinatorial data of the arrangement. If this system has a solution over $\C$, then it has a solution in $\overline{\Q}$ by Hilbert's Nullstellensatz. 

We finalise with the following strengthening of Conjecture \ref{complex}.

\begin{conjecture}
The set of accumulation points of Chern slopes of arrangements over $\Q$ is $[2,\frac{5}{2}]$.
\label{rationals}
\end{conjecture}

\subsection*{Acknowledgements} 

First author funded by CONICYT PFCHA / Doctorado Becas Chile/2015 - 72160240. Second author funded by CONICYT-PFCHA /Mag\'ister Nacional/2018 - 22180988. The third author was supported by the FONDECYT regular grant 1190066.


\end{document}